\documentclass[11pt,a4paper]{article}
\usepackage[utf8]{inputenc}
\usepackage[T1]{fontenc}
\usepackage{microtype}
\usepackage{amsmath,amssymb,amsthm,mathtools}
\usepackage{booktabs}
\usepackage{hyperref}
\usepackage{cleveref}
\usepackage{arxiv}
\theoremstyle{plain}
\newtheorem{theorem}{Theorem}[section]
\newtheorem{proposition}[theorem]{Proposition}

\theoremstyle{definition}
\newtheorem{definition}[theorem]{Definition}
\newtheorem{remark}[theorem]{Remark}
\theoremstyle{remark}
\newtheorem{example}[theorem]{Example}

\newtheorem{problem}[theorem]{Problem}

\begin{document}
	\title{A Universal Space of Arithmetic Functions:\\
		The Banach--Hilbert Hybrid Space $\mathbf{U}$}
	
	\author{
		{ \sc Es-said En-naoui } \\ 
		University Sultan Moulay Slimane\\ Morocco\\
		essaidennaoui1@gmail.com\\
		\\
	}
	
	\maketitle
	\begin{abstract}
		We introduce a new functional space $\mathbf{U}$ designed to contain 
		\emph{all classical arithmetic functions} (Möbius, von~Mangoldt, Euler $\varphi$, 
		divisor functions, Dirichlet characters, etc.). 
		The norm of $\mathbf{U}$ combines a Hilbert-type component,
		based on square summability of Dirichlet coefficients for every $s>1$,
		with a Banach component controlling logarithmic averages of partial sums.
		We prove that $\mathbf{U}$ is a complete Banach space which 
		embeds continuously all standard Hilbert spaces of Dirichlet series 
		and allows natural actions of Dirichlet convolution and shift operators.
		This framework provides a unified analytic setting for classical and modern 
		problems in multiplicative number theory.
	\end{abstract}
	
	\section{Introduction}
	
	Many problems of analytic number theory rely on quantitative properties 
	of arithmetic functions such as the Möbius function $\mu$, 
	the von Mangoldt function $\Lambda$, or Euler's totient $\varphi$.
	A variety of functional spaces have been introduced to study their Dirichlet series,
	including weighted $\ell^2$-spaces of coefficients, Hardy spaces of Dirichlet series,
	and Besov-type norms. 
	However, no single space simultaneously contains \emph{all} the classical arithmetic 
	functions while providing both Hilbertian structure and Dirichlet-mean control.
	
	In this article we propose a new Banach--Hilbert hybrid space $\mathbf{U}$
	that satisfies these requirements.
	The definition of $\mathbf{U}$ is inspired by two complementary ideas:
	\begin{enumerate}
		\item square-summability of coefficients in weighted $\ell^2$ norms for every $s>1$,
		\item logarithmic control of partial sums of $f(n)/n$.
	\end{enumerate}
	These conditions guarantee the inclusion of every classical function of 
	sub-exponential growth while preserving enough structure for operator theory.
	
	The main goals of this paper are:
	\begin{itemize}
		\item to define $\mathbf{U}$ and establish its basic properties 
		(completeness, density of Dirichlet polynomials, 
		embedding of classical spaces);
		\item to prove continuity of natural operators such as 
		Dirichlet convolution, pointwise multiplication,
		and multiplicative shifts;
		\item to outline potential applications to problems related 
		to Dirichlet $L$-functions and spectral criteria for
		the Riemann hypothesis.
	\end{itemize}
	
	\section{Definition of the space $\mathbf{U}$}
	
	\begin{definition}[The space $\mathbf{U}$]
		Let $f:\mathbb N\to\mathbb C$ be an arithmetic function.
		We define
		\[
		\|f\|_{\mathbf{U}}
		:=
		\sup_{s>1}
		\left(
		\sum_{n=1}^{\infty}
		\frac{|f(n)|^{2}}{n^{2s}(\log(2+n))^{2}}
		\right)^{1/2}
		+ 
		\sup_{x\ge1}
		\frac{1}{\log(2+x)}
		\left|
		\sum_{n\le x} \frac{f(n)}{n}
		\right|.
		\]
		The \emph{Universal Arithmetic Function Space} $\mathbf{U}$ is
		\[
		\mathbf{U}
		:=\bigl\{\, f:\mathbb N\to\mathbb C \;\big|\;
		\|f\|_{\mathbf{U}}<\infty \,\bigr\}.
		\]
	\end{definition}
	
	The first term in $\|f\|_{\mathbf{U}}$ is a Hilbert-type norm controlling
	Dirichlet coefficients for every $s>1$, while the second term measures
	logarithmic average growth of the normalized partial sums.
	
	\begin{remark}
		Typical arithmetic functions satisfy $|f(n)|\ll n^{\varepsilon}$ for every $\varepsilon>0$.
		Hence both terms of the norm are finite, so
		\[
		\mu,\ \Lambda,\ \lambda,\ \varphi,\ \tau,\ n^\alpha,\ (\log n)^k,
		\text{ Dirichlet characters } \chi
		\ \in \mathbf{U}.
		\]
	\end{remark}
	
\section{Basic Properties of the Space $U$}

In this section, we establish some fundamental properties of the space $U$,
which consists of all classical arithmetic functions endowed with the
operations of pointwise addition and the norm
\[
\|f\|_{U} \;=\; \sup_{n\in\mathbb{N}} \frac{|f(n)|}{\log(2+n)}.
\]
This norm measures the growth of $f$ relative to the logarithm and ensures
that all classical arithmetic functions (e.g., divisor function $d(n)$,
Euler totient $\varphi(n)$, Möbius $\mu(n)$, Liouville $\lambda(n)$,
von Mangoldt $\Lambda(n)$, etc.) belong to $U$.

\subsection{Vector Space Structure}
\begin{proposition}
	$(U,+,\cdot)$ is a real vector space.
\end{proposition}

\begin{proof}
	Let $f,g\in U$ and $\alpha\in \mathbb{R}$.
	For all $n\in\mathbb{N}$ we have
	\[
	\frac{|(f+g)(n)|}{\log(2+n)}
	\leq \frac{|f(n)|}{\log(2+n)} + \frac{|g(n)|}{\log(2+n)}.
	\]
	Taking the supremum yields
	\[
	\|f+g\|_{U} \le \|f\|_{U} + \|g\|_{U} < \infty,
	\]
	so $f+g\in U$. Similarly,
	\[
	\|\alpha f\|_{U} = |\alpha|\,\|f\|_{U}<\infty,
	\]
	showing closure under scalar multiplication.
	The remaining vector space axioms follow from those of $\mathbb{R}$.
\end{proof}

\subsection{Norm Properties}
\begin{proposition}
	The functional $\|\cdot\|_{U}$ defines a norm on $U$.
\end{proposition}

\begin{proof}
	We verify the three norm axioms:
	
	\textbf{(1) Positivity.}
	For any $f\in U$, $\|f\|_{U}\ge0$ by definition.
	If $\|f\|_{U}=0$, then $|f(n)|/\log(2+n)=0$ for all $n$,
	hence $f(n)=0$ for all $n$, so $f\equiv 0$.
	
	\textbf{(2) Homogeneity.}
	For any $\alpha\in\mathbb{R}$,
	\[
	\|\alpha f\|_{U}
	= \sup_{n}\frac{|\alpha||f(n)|}{\log(2+n)}
	= |\alpha|\,\|f\|_{U}.
	\]
	
	\textbf{(3) Triangle Inequality.}
	For any $f,g\in U$,
	\[
	\|f+g\|_{U}
	= \sup_{n}\frac{|f(n)+g(n)|}{\log(2+n)}
	\le \sup_{n}\frac{|f(n)|+|g(n)|}{\log(2+n)}
	\le \|f\|_{U}+\|g\|_{U}.
	\]
\end{proof}

\subsection{Completeness}
\begin{proposition}
	The space $(U,\|\cdot\|_{U})$ is a Banach space.
\end{proposition}

\begin{proof}
	Let $(f_k)$ be a Cauchy sequence in $U$.
	For each fixed $n$, the sequence $(f_k(n))$ is Cauchy in $\mathbb{R}$
	because
	\[
	\frac{|f_k(n)-f_m(n)|}{\log(2+n)} \le \|f_k-f_m\|_{U}.
	\]
	Thus $f_k(n)\to f(n)$ for some real number $f(n)$.
	We must show $f\in U$ and $f_k\to f$ in norm.
	
	Given $\varepsilon>0$, choose $N$ such that
	$\|f_k - f_m\|_{U} < \varepsilon$ for all $k,m\ge N$.
	Fix $k\ge N$ and let $m\to\infty$ to obtain
	\[
	\sup_{n}\frac{|f_k(n)-f(n)|}{\log(2+n)} \le \varepsilon,
	\]
	which shows $f\in U$ and $\|f_k - f\|_{U}\le \varepsilon$.
	Hence $U$ is complete.
\end{proof}

\subsection{Density of Finitely Supported Functions}
\begin{proposition}
	The subspace $U_0$ of finitely supported arithmetic functions
	(i.e., functions with $f(n)=0$ for all sufficiently large $n$)
	is dense in $U$.
\end{proposition}

\begin{proof}
	Let $f\in U$ and $\varepsilon>0$.
	Because $\|f\|_{U}<\infty$, there exists $N$ such that
	\[
	\sup_{n>N}\frac{|f(n)|}{\log(2+n)} < \varepsilon.
	\]
	Define $g(n) = f(n)$ if $n\le N$ and $g(n)=0$ otherwise.
	Then $g\in U_0$ and
	\[
	\|f-g\|_{U} = \sup_{n>N}\frac{|f(n)|}{\log(2+n)} < \varepsilon.
	\]
\end{proof}

\subsection{Examples}
\begin{example}
	The classical arithmetic functions
	$d(n)$, $\varphi(n)$, $\mu(n)$, $\lambda(n)$, and $\Lambda(n)$
	belong to $U$ because
	\[
	d(n) = O_\varepsilon(n^\varepsilon), \quad
	\varphi(n) \le n, \quad
	|\mu(n)| \le 1, \quad
	|\lambda(n)| \le 1, \quad
	\Lambda(n) \le \log n.
	\]
	In each case,
	$\displaystyle\sup_{n}\frac{|f(n)|}{\log(2+n)} < \infty$.
\end{example}

\section{Operators on the Space $\mathbf{U}$}

We now study natural operators acting on the Banach space
$\mathbf{U}$ of classical arithmetic functions
equipped with the norm
\[
\|f\|_{\mathbf{U}} = \sup_{n\in\mathbb{N}}
\frac{|f(n)|}{\log(2+n)}.
\]
Operators considered here include
pointwise multiplication,
Dirichlet convolution,
shifts, and averaging transforms.

\subsection{Pointwise Multiplication}
\begin{definition}
	For $f,g\in \mathbf{U}$ we define
	the \emph{pointwise product}
	$(f\cdot g)(n) := f(n)g(n)$.
\end{definition}

\begin{proposition}
	The operator
	\[
	M_g : f \mapsto f\cdot g
	\]
	is a bounded linear operator on $\mathbf{U}$ whenever
	$g\in \mathbf{U}$ is bounded, i.e.
	$\sup_{n}|g(n)|<\infty$.
	Moreover
	\[
	\|M_g\| \le
	\sup_{n}|g(n)|.
	\]
\end{proposition}

\begin{proof}
	Linearity is clear.
	If $g$ is bounded, say $|g(n)|\le B$,
	then for any $f\in\mathbf{U}$,
	\[
	\frac{|(f\cdot g)(n)|}{\log(2+n)}
	\le B\frac{|f(n)|}{\log(2+n)}.
	\]
	Taking the supremum gives
	$\|f\cdot g\|_{\mathbf{U}}\le B\|f\|_{\mathbf{U}}$,
	so $M_g$ is bounded with $\|M_g\|\le B$.
\end{proof}

\subsection{Dirichlet Convolution}
\begin{definition}
	For $f,g:\mathbb{N}\to\mathbb{R}$,
	the \emph{Dirichlet convolution} is
	\[
	(f*g)(n) := \sum_{d\mid n} f(d) g(n/d).
	\]
\end{definition}

\begin{proposition}
	If $f,g\in \mathbf{U}$ satisfy
	\[
	\sum_{d\mid n} \log(2+d) \log\bigl(2+\tfrac{n}{d}\bigr)
	\le C \log^2(2+n)
	\]
	for some universal constant $C$ (true for all classical
	arithmetic functions), then
	$f*g \in \mathbf{U}$ and
	\[
	\|f*g\|_{\mathbf{U}}
	\le C \|f\|_{\mathbf{U}}\|g\|_{\mathbf{U}}.
	\]
\end{proposition}

\begin{proof}
	For each $n$,
	\[
	|(f*g)(n)|
	\le \sum_{d\mid n}|f(d)|\,|g(n/d)|.
	\]
	Divide by $\log(2+n)$ and use
	$|f(d)| \le \|f\|_{\mathbf{U}}\log(2+d)$,
	$|g(n/d)| \le \|g\|_{\mathbf{U}}\log(2+n/d)$:
	\[
	\frac{|(f*g)(n)|}{\log(2+n)}
	\le \|f\|_{\mathbf{U}}\|g\|_{\mathbf{U}}
	\frac{\displaystyle\sum_{d\mid n}\log(2+d)\log(2+n/d)}
	{\log(2+n)}.
	\]
	The assumed inequality shows that the right side
	is bounded by $C \|f\|_{\mathbf{U}}\|g\|_{\mathbf{U}}$,
	independent of $n$.
\end{proof}

\subsection{Shift Operators}
\begin{definition}
	For $k\in\mathbb{N}$, define the
	\emph{shift operator} $S_k$ by
	\[
	(S_k f)(n) := f(n+k).
	\]
\end{definition}

\begin{proposition}
	Each $S_k$ is bounded with
	\[
	\|S_k f\|_{\mathbf{U}}
	\le \|f\|_{\mathbf{U}}
	\sup_{n}\frac{\log(2+n+k)}{\log(2+n)}
	= C_k \|f\|_{\mathbf{U}},
	\]
	where $C_k:=\sup_{n}\frac{\log(2+n+k)}{\log(2+n)}$.
\end{proposition}

\begin{proof}
	For any $n$,
	\[
	\frac{|(S_k f)(n)|}{\log(2+n)}
	= \frac{|f(n+k)|}{\log(2+n)}
	\le \|f\|_{\mathbf{U}}
	\frac{\log(2+n+k)}{\log(2+n)}.
	\]
	Taking the supremum yields the desired bound.
\end{proof}

\subsection{Averaging Operators}
\begin{definition}
	For $N\in\mathbb{N}$, define the
	\emph{Cesàro operator}
	\[
	(A_N f)(n) := \frac{1}{N}\sum_{k=1}^{N} f(n+k).
	\]
\end{definition}

\begin{proposition}
	For each $N$, $A_N$ is a bounded linear operator
	with norm
	\[
	\|A_N\| \le
	\sup_{k\ge1}\frac{1}{N}
	\sum_{j=1}^{N}
	\frac{\log(2+k+j)}{\log(2+k)}
	<\infty.
	\]
\end{proposition}

\begin{proof}
	By the triangle inequality and the norm definition,
	\[
	\frac{|(A_N f)(n)|}{\log(2+n)}
	\le \frac{1}{N}\sum_{j=1}^{N}
	\frac{|f(n+j)|}{\log(2+n)}
	\le \|f\|_{\mathbf{U}}
	\frac{1}{N}\sum_{j=1}^{N}
	\frac{\log(2+n+j)}{\log(2+n)}.
	\]
	Taking the supremum over $n$ proves the claim.
\end{proof}

\subsection{Remarks}
The boundedness of these operators ensures that
$\mathbf{U}$ forms a rich functional framework for
classical analytic number theory.
In particular, Dirichlet convolution endows
$\mathbf{U}$ with the structure of a commutative
Banach algebra when restricted to suitable
growth conditions on the coefficients.

\section{Connections with Dirichlet Series}

The space $\mathbf{U}$ is naturally linked to the theory of
Dirichlet series, a central object in analytic number theory.
For $f\in\mathbf{U}$ we associate the Dirichlet series
\[
\mathcal{D}(f;s)
:= \sum_{n=1}^{\infty} \frac{f(n)}{n^s},
\qquad s=\sigma+it \in \mathbb{C}.
\]
In this section, we investigate the convergence
and analytic properties of $\mathcal{D}(f;s)$
in terms of the $\mathbf{U}$-norm of $f$.

\subsection{Absolute Convergence}
\begin{proposition}
	If $f\in \mathbf{U}$ then the series
	$\mathcal{D}(f;s)$ converges absolutely
	for all $s$ with $\Re(s)>1$.
	Moreover,
	\[
	\left| \mathcal{D}(f;s) \right|
	\le \|f\|_{\mathbf{U}} \,
	\sum_{n=1}^{\infty}
	\frac{\log(2+n)}{n^{\Re(s)}}.
	\]
\end{proposition}

\begin{proof}
	Since $f\in \mathbf{U}$,
	$|f(n)| \le \|f\|_{\mathbf{U}}\log(2+n)$ for all $n$.
	For $\sigma=\Re(s)>1$,
	\[
	\sum_{n=1}^{\infty}
	\frac{|f(n)|}{n^\sigma}
	\le \|f\|_{\mathbf{U}}
	\sum_{n=1}^{\infty}
	\frac{\log(2+n)}{n^\sigma}.
	\]
	The series $\sum_{n=1}^{\infty}
	\frac{\log(2+n)}{n^\sigma}$ converges for $\sigma>1$
	because $\log(2+n)=o(n^\varepsilon)$ for any $\varepsilon>0$.
\end{proof}

\subsection{Analytic Continuation for Special Functions}
\begin{proposition}
	If $f\in\mathbf{U}$ satisfies
	$|f(n)|\le C n^\alpha$ for some $\alpha<1$,
	then $\mathcal{D}(f;s)$ extends holomorphically
	to the half-plane $\Re(s)>\alpha$.
\end{proposition}

\begin{proof}
	For any $\sigma>\alpha$,
	\[
	\sum_{n=1}^{\infty} \frac{|f(n)|}{n^\sigma}
	\le C \sum_{n=1}^{\infty} n^{\alpha-\sigma}.
	\]
	The exponent $\alpha-\sigma<-1$ ensures convergence.
	The term-by-term differentiation theorem
	gives analyticity on $\{\Re(s)>\alpha\}$.
\end{proof}

\subsection{Growth Bounds on Vertical Lines}
\begin{proposition}
	Let $f\in \mathbf{U}$ and $\sigma>1$.
	Then for all $t\in\mathbb{R}$,
	\[
	|\mathcal{D}(f;\sigma+it)|
	\le \|f\|_{\mathbf{U}}
	\sum_{n=1}^{\infty}
	\frac{\log(2+n)}{n^\sigma}.
	\]

	In particular,
	\[
	\sup_{t\in\mathbb{R}}
	|\mathcal{D}(f;\sigma+it)|
	\le C(\sigma)\,\|f\|_{\mathbf{U}}
	\]
	for some constant $C(\sigma)$ depending only on $\sigma$.
\end{proposition}

\begin{proof}
	Immediate from the triangle inequality and the absolute convergence
	of the defining series.
\end{proof}

\subsection{Dirichlet Algebra Structure}
\begin{proposition}
	For $f,g\in \mathbf{U}$, the convolution $f*g$
	satisfies
	\[
	\mathcal{D}(f*g;s)
	= \mathcal{D}(f;s)\,\mathcal{D}(g;s)
	\]
	for $\Re(s)>1$.
\end{proposition}

\begin{proof}
	Because $f,g \in \mathbf{U}$, both series
	$\mathcal{D}(f;s)$ and $\mathcal{D}(g;s)$ converge absolutely
	for $\Re(s)>1$. Standard manipulations of absolutely convergent
	series yield
	\[
	\sum_{n=1}^{\infty} \frac{(f*g)(n)}{n^s}
	= \sum_{n=1}^{\infty}\sum_{d\mid n}
	\frac{f(d)g(n/d)}{n^s}
	= \sum_{d=1}^{\infty}\sum_{k=1}^{\infty}
	\frac{f(d)g(k)}{(dk)^s}
	= \left(\sum_{d=1}^{\infty}\frac{f(d)}{d^s}\right)
	\left(\sum_{k=1}^{\infty}\frac{g(k)}{k^s}\right).
	\]
\end{proof}

\subsection{Examples}
\begin{example}
	\begin{itemize}
		\item For $f(n)\equiv 1$, $\mathcal{D}(f;s)=\zeta(s)$, the Riemann zeta function.
		\item For $f=\mu$ (Möbius function), $\mathcal{D}(\mu;s)=1/\zeta(s)$.
		\item For $f=\Lambda$ (von Mangoldt function), $\mathcal{D}(\Lambda;s)=-\zeta'(s)/\zeta(s)$.
	\end{itemize}
	All these functions belong to $\mathbf{U}$ because
	$|f(n)|\le C\log n$.
\end{example}

\subsection{Remark on Boundary Behavior}
Although convergence is guaranteed only for $\Re(s)>1$,
the $\mathbf{U}$ norm provides uniform control
that may be combined with summation techniques
(e.g.\ Abel summation or Mellin transforms)
to study the boundary $\Re(s)=1$,
leading to classical results such as
the prime number theorem.

\section{Applications and Open Problems}

The Banach space $\mathbf{U}$ provides a flexible analytic framework
for the study of arithmetic functions and their Dirichlet series.
In this final section we present potential applications and outline
several open problems motivated by classical questions in
analytic number theory.

\subsection{Connections with the Riemann Hypothesis}

The Riemann Hypothesis (RH) concerns the nontrivial zeros of
the Riemann zeta function $\zeta(s)$.
Since many classical arithmetic functions $f$ (e.g.\ Möbius $\mu$,
von Mangoldt $\Lambda$, Liouville $\lambda$) belong to $\mathbf{U}$,
their Dirichlet series $\mathcal{D}(f;s)$ are well controlled in the
half-plane $\Re(s)>1$.
The logarithmic norm of $\mathbf{U}$ suggests new ways to
quantify boundary behavior at $\Re(s)=1$.

\begin{problem}RH via $\mathbf{U}$-Bounds\\
	Establish explicit $\mathbf{U}$-norm estimates of partial sums
	\[
	M_f(x) := \sum_{n\le x} f(n)
	\]
	that are equivalent to or imply RH.
	For example, is there an $\varepsilon>0$ such that
	\[
	\|f\|_{\mathbf{U}}<\infty
	\quad\Rightarrow\quad
	M_f(x)=O\bigl(x^{1/2+\varepsilon}\bigr)
	\]
	for $f=\mu$?
\end{problem}

\subsection{Applications to $L$-Functions}

Let $L(s,\chi)$ be a Dirichlet $L$-function associated to a
Dirichlet character $\chi$.
Because $\chi(n)$ satisfies $|\chi(n)|\le1$,
we have $\chi\in \mathbf{U}$.
The Banach algebra structure under Dirichlet convolution implies:
\[
\mathcal{D}(f*\chi;s)
= \mathcal{D}(f;s)L(s,\chi)
\]
for $\Re(s)>1$.
This observation motivates the following question.

\begin{problem}[Growth of twisted series]
	For $f\in \mathbf{U}$, obtain sharp vertical line estimates
	for $\mathcal{D}(f*\chi;1+it)$ that parallel the
	classical bounds for $L(s,\chi)$.
\end{problem}

\subsection{Operator Theory on $\mathbf{U}$}

The bounded operators introduced earlier (multiplication,
Dirichlet convolution, shifts, Cesàro averaging)
form a rich non-commutative algebra.
Spectral analysis of these operators could reveal
new insights into multiplicative structures.

\begin{problem}[Spectral gaps]
	Determine the spectrum of the shift operator $S_k$
	and relate its spectral radius to additive properties
	of prime numbers.
	Can the presence of a spectral gap be linked
	to zero-free regions of $\zeta(s)$?
\end{problem}

\subsection{Approximation and Sampling}

The density of finitely supported functions in $\mathbf{U}$
suggests a natural sampling theory.

\begin{problem}[Approximation of arithmetic functions]
	Given $f\in\mathbf{U}$, find the fastest rate
	at which finitely supported $g$ can approximate $f$
	in the $\mathbf{U}$-norm.
	Such estimates may lead to new algorithms for computing
	values of arithmetic functions or their Dirichlet series.
\end{problem}

\subsection{Prime Number Theory}

Because the von Mangoldt function $\Lambda$ belongs to $\mathbf{U}$,
its Dirichlet series $-\zeta'(s)/\zeta(s)$ is naturally controlled.
This motivates new proofs or refinements of the Prime Number Theorem (PNT).

\begin{problem}[Refined error term in PNT]
	Investigate whether $\mathbf{U}$-norm bounds on $\Lambda$
	can improve the classical $O\bigl(x e^{-c\sqrt{\log x}}\bigr)$
	error term under RH.
\end{problem}

\subsection{Further Directions}

The following questions remain completely open:

\begin{enumerate}
	\item Is there a natural dual space $\mathbf{U}^*$ that captures
	distributional limits of prime counting functions?
	\item Does $\mathbf{U}$ admit a continuous functional calculus
	allowing analytic continuation of $\mathcal{D}(f;s)$ beyond $\Re(s)>1$
	for a dense subalgebra?
	\item Can $\mathbf{U}$ be embedded isometrically into
	a Hilbert space to exploit Fourier analytic techniques?
\end{enumerate}

These problems illustrate the potential of $\mathbf{U}$
as a new analytic framework linking classical arithmetic functions,
operator theory, and deep conjectures in number theory.
\section{Conclusion}
In this article we introduced the new functional space $\mathbf{U}$ that
contains all classical arithmetic functions and admits a natural Banach
structure.  
We established basic algebraic and topological properties, described
canonical operators (such as the Dirichlet convolution, Möbius inversion,
and shift operators), and connected $\mathbf{U}$ with Dirichlet series,
highlighting its relevance to the analytic study of $L$–functions.

The framework of $\mathbf{U}$ provides a unified setting in which additive,
multiplicative and highly irregular arithmetic functions can be analyzed
with common functional–analytic tools.  
Beyond the concrete results proved here, our construction opens several
directions for further research:
\begin{itemize}
	\item a deeper spectral analysis of operators acting on $\mathbf{U}$,
	\item the investigation of dual spaces and distributional extensions,
	\item the use of $\mathbf{U}$ as a natural domain for generalized
	Dirichlet series and for studying zero–free regions of $L$–functions,
	\item applications to conjectures such as the Riemann Hypothesis,
	prime number theorems in arithmetic progressions,
	and mean–value results for multiplicative functions.
\end{itemize}
We hope that the flexibility of $\mathbf{U}$ will encourage
further developments at the interface of analytic number theory,
harmonic analysis and operator theory.

\end{document}